\documentclass[12pt]{article}
\usepackage{amsmath}
\usepackage[affil-it]{authblk}
\usepackage{amsfonts}
\usepackage{amsthm}
\usepackage{amssymb}
\usepackage{makecell}
\usepackage{appendix}
\usepackage{algorithmic}
\usepackage{caption}
\usepackage{graphicx}
\usepackage{algorithm2e}
\usepackage{tabu}
\usepackage{color}
\usepackage{soul}
\usepackage{url}
\definecolor{darkgreen}{rgb}{0.0, 0.63, 0.0}
\theoremstyle{definition}

\usepackage{mathtools}

\theoremstyle{lemma}
\newtheorem{theorem}{Theorem}
\newtheorem{corollary}{Corollary}[theorem]

\newtheorem{proposition}{Proposition}
\newtheorem{conjecture}{Conjecture}

\newcommand{\norm}{\text{Nm}}

\newcommand{\trace}{\text{Tr}}

\usepackage[margin=0.9in]{geometry}

\title{Unit Reducible Cyclotomic Fields}
\author{Christian Porter, Piero Sarti, Cong Ling, Alar Leibak}
\date{July 2023}

\begin{document}

\maketitle
\begin{abstract}
    In this paper, we continue the study of unit reducible fields as introduced in \cite{LPL23} for the special case of cyclotomic fields. Specifically, we deduce that the cyclotomic fields of conductors $2,3,5,7,8,9,12,15$ are all unit reducible, and show that any cyclotomic field of conductor $N$ is not unit reducible if $2^4, 3^3, 5^2, 7^2, 11^2$ or any prime $p \geq 13$ divide $N$, meaning the unit reducible cyclotomic fields are finite in number. Finally, if $a$ is a totally positive element of a cyclotomic field, we show that for all equivalent $a^\prime$, the discrepancy between $\trace_{K/\mathbb{Q}}(a^\prime)$ and the shortest nonzero element of the quadratic form $\trace_{K/\mathbb{Q}}(axx^*)$ where $x$ is taken from the ring of integers tends to infinity as the conductor $N$ goes to infinity.
\end{abstract}
\section{Introduction}
Let $K$ be a Galois extension of $\mathbb{Q}$ of degree $n$, where $n$ is either $r$ if $K$ is totally real or $2s$ if $K$ is totally complex, for some integers $r$ or $2s$. Associate to $K$ the canonical embeddings $\sigma_1,\dots, \sigma_n$ into $\mathbb{R}$ if $K$ is totally real, or into $\mathbb{C}$, if $K$ is totally complex. We define a quadratic form $f:K^m \to K$ by
\begin{align*}
    f(x_1,\dots,x_m)=\sum_{k,l=1}^m f_{k,l}x_kx_l^*,
\end{align*}
where $f_{k,l} \in K$ and $*$ denotes complex conjugation if $K$ is totally complex, or fixes $x_l$ if $K$ is totally real. Usually, we take $f_{k,l}=f_{l,k}^*$ for all $k,l$. $f$ is said to be positive definite if
\begin{align*}
    \sigma_i(f)(x_1,\dots,x_m)=\sum_{k,l} \sigma_i(f_{k,l})\sigma_i(x_k)\sigma_i(x_l)^*
\end{align*}
is a real positive definite quadratic form or positive definite hermitian form, for all $1 \leq i \leq n$. From now on, we will assume all quadratic forms are positive-definite and simply refer to them as quadratic forms.

A quadratic form $f$ is called a unary form if it is defined in a single variable in $K$, i.e. the quadratic form takes the form $axx^*$. We say that the element $a$ is totally positive if $\sigma_i(a) \in \mathbb{R}_{>0}$ for all $1 \leq i \leq n$.  We denote the set of totally positive elements of $K$ by $K_{>>0}$. Let $\mathcal{O}_K$ and $\mathcal{O}_K^*$ respectively denote the ring of integers and unit group of $K$. Then for any $a \in K_{>>0}$, we will use the notation
\begin{align*}
    &\mu(a)=\min_{x \in \mathcal{O}_{K} \setminus \{0\}}\trace_{K/\mathbb{Q}}(axx^*),
    \\&\mu^*(a)=\min_{u \in \mathcal{O}_K^*} \trace_{K/\mathbb{Q}}(auu^*),
    \\& \mathcal{M}(a)=\{x \in \mathcal{O}_K: \trace_{K/\mathbb{Q}}(axx^*)=\mu(a)\}.
\end{align*}
We say that $a$ is reduced if $\trace_{K/\mathbb{Q}}(auu^*) \geq \trace_{K/\mathbb{Q}}(a)(=\mu^*(a))$ for all $u \in \mathcal{O}_K^*$. Note that the set $\{auu^*: u \in \mathcal{O}_K^*\}$ induces an equivalence class for each totally positive $a$. The set $\mathcal{F}_K$ denotes the set of all reduced totally positive elements of $K$, which we refer to as the reduction domain of $K$. 

The field $K$ is unit reducible if $\mu(a)=\mu^*(a)$ for all $a \in K_{>>0}$ - that is, every totally positive $a$ attains its trace-minimum at a unit. In \cite{LPL23}, it was shown that a real quadratic field $\mathbb{Q}(\sqrt{d})$ is unit reducible if and only if $d$ took one of the following forms:
\begin{align*}
    &d=m^2+1, \hspace{1mm} m \in \mathbb{N}, \hspace{1mm} m \hspace{1mm} \text{odd},
    \\&d=m^2-1, \hspace{1mm} m \in \mathbb{N}, \hspace{1mm} m \hspace{1mm} \text{even},
    \\&d=m^2+4, \hspace{1mm} m \in \mathbb{N}, \hspace{1mm} m \hspace{1mm} \text{even},
    \\&d=m^2-4, \hspace{1mm} m \in \mathbb{N}, \hspace{1mm} m>3 \hspace{1mm} \text{odd}.
\end{align*}
Moreover, it was shown that the cubic field with defining polynomial
\begin{align*}
    x^3-tx^2-(t+3)x-1, \hspace{1mm}  t \in \mathbb{Z}_{\geq 0},
\end{align*}
is unit reducible if the field is monogenic. Specifically, we will call a unit reducible field $K$ strongly unit reducible if $\mathcal{M}(a) \subset \mathcal{O}_K^*$ for all totally positive $a$ and weakly unit reducible field otherwise.

We define by $\zeta_N=\exp(2\pi i/N)$ a primitive $N$th root of unity. We call the field $K=\mathbb{Q}(\zeta_N)$ the cyclotomic field of conductor $N$, which has degree $\phi(N)$ over $\mathbb{Q}$. We will use the notation $K_N=\mathbb{Q}(\zeta_N)$ from now on. 

We define the reduction discrepancy of a field by the quantity
\begin{align*}
    \delta_K=\sup_{a \in K_{>>0}} \frac{\mu^*(a)}{\mu(a)}.
\end{align*}
Clearly $\delta_K \geq 1$ for every field $K$, and if $\delta_K=1$ then $K$ is unit reducible.

Finally, the Hermite-Humbert constant of a field $K$ is defined by
\begin{align*}
    \gamma_K=\sup_{a \in K_{>>0}}\frac{\mu(a)}{\norm_{K/\mathbb{Q}}(a)^{1/n}}.
\end{align*}
 For any positive definite real quadratic form $q(x_1,\dots,x_n)$ of rank $n$, we define by
\begin{align*}
    \gamma(q)=\frac{\min_{(x_1,\dots,x_n) \in \mathbb{Z}^n \setminus (0,\dots,0)}q(x_1,\dots,x_n)}{\det(q)^{1/n}},
\end{align*}
where $\det(q)$ is the determinant of the matrix $Q=\{q_{i,j}\}_{1 \leq i,j \leq n}$, where $q(x_1,\dots,x_n)=\sum_{i,j=1}^n q_{i,j}x_ix_j$. Then the Hermite constant $\gamma_n=\max_q \gamma(q)$, where the maximum is taken over all positive definite real quadratic forms $q$.

\subsection{Motivation}
In this paper, we are concerned with unit reducible cyclotomic fields. Cyclotomic fields are of significant importance in the field of post-quantum cryptography, underpinning the security and enjoying efficient implementation of mainstream lattice-based cryptosystems (see e.g. \cite{LPR,SS}). In July 2022, the US National Institute of Standards and Technology (NIST) announced the selection of four post-quantum cryptographic algorithms to be standardized, three of which are based on cyclotomic fields \cite{NIST_Round3}. Lattice-based cryptography relies on the hardness of the ``shortest vector problem" (SVP), where an adversary tries to find a shortest nonzero vector in a given lattice. A common strategy of attack is to run lattice reduction over these number fields to find short vectors. Some cyclotomic lattice cryptosystems make use of principal ideals (see e.g. \cite{SV10,CGS14,GGH13,LSS14}). The study of principal ideals as lattices has a natural correspondence with the theory of unary forms. Finding the smallest nonzero value of a unary quadratic form gives the shortest vector of such a lattice. 

Specifically, the recovery of short generators of principal ideals in cyclotomic fields was studied in \cite{CDPR16}, where the length of ideal elements was measured by their length under the canonical embedding. Let $\mathcal{I}$ be a principal ideal in some number field $K$ with ring of integers $\mathcal{O}_K$. If we denote by $g$ a generator of $\mathcal{I}$, then solving the shortest generator problem (as in the setting of \cite{CDPR16}) corresponds to finding a unit $u \in \mathcal{O}_K^*$ such that
\begin{align*}
    \trace_{K/\mathbb{Q}}(auu^*)=\mu^*(a),
\end{align*}
where $a=gg^*$, and solving the shortest vector problem in $\mathcal{I}$ corresponds to finding $x \in \mathcal{O}_K$ such that
\begin{align*}
    \trace_{K/\mathbb{Q}}(axx^*)=\mu(a).
\end{align*}
Though the reduction discrepancy was not formally defined, it was shown that for any prime power $M$,
\begin{align*}
    \exp\left(\Omega\left(\frac{\sqrt{M}}{h^\prime \log(M)}\right)\right)\leq \delta_{K_M} \leq \phi(M)\exp(O(\sqrt{M \log(M)}),
\end{align*}
where $h^\prime={h^+}^{\frac{1}{\varphi(M)-1}}$, and $h^+$ is the class number of the maximal totally real subfield of $K_M$. These seem to be the best asymptotic bounds for large values of $M$.

When the field has a special ``unit reducible" property, the smallest quadratic form value also gives the shortest generator of the ideal. Although this connection breaks down in high dimensions as shown above, it can well hold in low dimensions. Our ultimate goal is to find a full list of cyclotomic fields which are unit reducible. Meanwhile, studying unit reducible fields in low dimensions could lead to lattice reduction algorithms in high dimensions (since the former can be used as a subroutine of the latter), an active research area \cite{LPL23,euclidean,lllmodules}. In summary, the unit reducible property connects the SVP to ideal generators in certain cyclotomic fields, with implications for cryptanalysis.

\subsection{Contributions}
The main results of the paper are as follows.
\begin{theorem}\label{1}
    For any number field $K$ such that $K/\mathbb{Q}$ is Galois, let
    \begin{align*}
        \mathcal{S}=\{x \in \mathcal{O}_K: |\norm_{K/\mathbb{Q}}(x)| \geq 2\} =\mathcal{O}_K\setminus \{\mathcal{O}_K^*\cup \{0\}\},
    \end{align*}
    and
    \begin{align*}
        \eta_K=\min_{x \in \mathcal{S}} |\norm_{K/\mathbb{Q}}(x)|.
    \end{align*}
    If $\gamma_K<\eta_K^{2/[K:\mathbb{Q}]}[K:\mathbb{Q}]$, then $K$ is strongly unit reducible.
\end{theorem}
\begin{theorem}\label{2}
    For any integers $M,N$ such that $N \mid M$, if $K_N$ is not unit reducible, the cyclotomic fields $K_M$ is also not unit reducible. Moreover, $\delta_{K_M} \geq \delta_{K_N}$.
\end{theorem}
We also find a more concrete lower bound for $\delta_{K_M}$ when $M$ is a prime power which, whilst asymptotically worse than that in \cite{CDPR16}, provides us a framework by which we can deduce whether or not such fields are unit reducible.
\begin{theorem}\label{3}
    The field $K_N$ is unit reducible for $N=3,4,5,7,8,9,12,15$. Of these, the fields of conductors $N=3,4,5,7,12,15$ are strongly unit reducible, whilst the fields of conductors $N=8,9$ are weakly unit reducible. Moreover, for all integers $n \geq 1$ and odd primes $p \geq 3$,
    \begin{align*}
        &\delta_{K_{2^n}} \geq 2^{n-3},
        \\& \delta_{K_{p^n}} \geq p^{n-1}\frac{p+1}{12}.
    \end{align*}
    
\end{theorem}
An immediate consequence of Theorems \ref{2} and \ref{3} are the following corollaries, the second of which we will prove, whilst the first is trivial.
\begin{corollary}
    The field $K_N$ is not unit reducible if $2^4, 3^3, 5^2, 7^2, 11^2$ or any prime $p \geq 13$ divides $N$.
\end{corollary}
\begin{corollary}\label{cor2}
    $\lim_{N \to \infty} \delta_{K_N}=\infty$.
\end{corollary}
\section{Unit Reducible Cyclotomic Fields}
Throughout the paper, we will use $\Delta_K$ to denote the discriminant of a number field $K$.
\begin{proof}[Proof of Theorem \ref{1}]
    By the definition of $\gamma_K$, every totally positive element $a \in K_{>>0}$,
    \begin{align*}
        \mu(a) \leq \gamma_K \norm_{K/\mathbb{Q}}(a)^{\frac{1}{n}},
    \end{align*}
    where $n=[K:\mathbb{Q}]$. Suppose that $\trace_{K/\mathbb{Q}}(axx^*)=\mu(a)$ for some nonzero $x \in \mathcal{O}_K$. By the arithmetic-geometric inequality,
    \begin{align*}
        n\norm_{K/\mathbb{Q}}(a)^{\frac{1}{n}}\norm_{K/\mathbb{Q}}(x)^{\frac{2}{n}} = n\norm_{K/\mathbb{Q}}(axx^*)^{\frac{1}{n}} \leq \trace_{K/\mathbb{Q}}(axx^*)=\mu(a) \leq \gamma_K\norm_{K/\mathbb{Q}}(a)^{\frac{1}{n}}.
    \end{align*}
    Rearranging gives
    \begin{align*}
        \norm_{K/\mathbb{Q}}(xx^*)^{\frac{2}{n}} \leq \frac{\gamma_K}{n}.
    \end{align*}
    If $n\eta_K^{\frac{2}{n}}>\gamma_K$, clearly the only nonzero elements of $\mathcal{O}_K$ that any totally positive $a \in K_{>>0}$ can attain their trace-minimum at must be units, which completes the proof.
\end{proof}
\begin{proof}[Proof of Theorem \ref{2}]
    Consider the field $K_N$ for some integer $N \geq 3$, and assume that it is not unit reducible. Then there exists an $a \in {K_N}_{>>0}$ such that $\mathcal{M}(a) \subset \mathcal{S}$, i.e. $\mu(a) <\mu^*(a)$. Let $N=sp^n$ for some $n \geq 0$ where $p$ is a prime number, and $\gcd(p,s)=1$. Now consider the field $K_M$, where $M=sp^m$ and $m \geq n \geq 0$. For any $y \in \mathcal{O}_{K_M}$, we can express $y$ as
    \begin{align}
        y=\sum_{i=0}^{p^{m-n}-1} x_i^\prime \zeta_M^i,
    \end{align}
    where $x_i \in \mathcal{O}_{K_N}$ and $x_i^\prime$ is the inclusion of $x_i$ into $K_M$. Let $a^\prime$ denote the element $a \in {K_N}_{>>0}$ after being lifted into the field $K_M$. Clearly $a^\prime$ is still a totally positive element. Then
    \begin{align*}
        a^{\prime}yy^*=\sum_{i,j=0}^{p^{m-n}-1}a^\prime x_i^\prime {x_j^{\prime}}^* \zeta_M^{i-j}.
    \end{align*}
    Using the fact that
    \begin{align*}
        &\trace_{K_M/\mathbb{Q}}\left(\sum_{i,j=0}^{p^{m-n}-1}a^\prime x_i^\prime {x_j^\prime}^* \zeta_M^{i-j}\right)=\sum_{i,j=0}^{p^{m-n}-1}\trace_{K_M/\mathbb{Q}}\left(a^\prime x_i^\prime {x_j^\prime}^* \zeta_M^{i-j}\right)
        \\ & = \sum_{i,j=0}^{p^{m-n}-1}\trace_{K_N/\mathbb{Q}}\left(a^\prime x_i^\prime {x_j^\prime}^* \ \trace_{K_M/K_N}(\zeta_M^{i-j})\right), \qquad\trace_{K_M/K_N}(\zeta_M^{i-j})=\begin{cases}
    p^{m-n} \hspace{2mm} &\text{if} \hspace{1mm} i=j,\\
    0 \hspace{2mm} &\text{otherwise.}
    \end{cases},
\end{align*}
    it follows that
    \begin{align}
        \trace_{K_M/\mathbb{Q}}(a^\prime yy^*)=p^{m-n}\sum_{i=0}^{p^{m-n}-1}\trace_{K_N/\mathbb{Q}}(ax_ix_i^*). \label{4}
    \end{align}
    We may assume without loss of generality that $a^\prime \in \mathcal{F}_{K_M}$, so
    \begin{align*}
        \trace_{K_M/\mathbb{Q}}(a^\prime) \leq \trace_{K_M/\mathbb{Q}}(a^\prime uu^*),
    \end{align*}
    for any $u \in \mathcal{O}_{K_M}^*$. However, by the assumptions made on $a$, there exists an element $x \in \mathcal{O}_{K_N}$ such that $x$ is not a unit or zero, and
    \begin{align*}
        \trace_{K_N/\mathbb{Q}}(avv^*) >\trace_{K_N/\mathbb{Q}}(axx^*),
    \end{align*}
    for all $v \in \mathcal{O}_{K_N}^*$. Therefore,
    \begin{align*}
        \trace_{K_M/\mathbb{Q}}(a^\prime)=p^{m-n}\trace_{K_N/\mathbb{Q}}(a)>p^{m-n}\trace_{K_N/\mathbb{Q}}(axx^*)=\trace_{K_M/\mathbb{Q}}(a^\prime x^\prime {x^\prime}^*),
    \end{align*}
    where $x^\prime$ is the element $x$ after being lifted into $K_M$, which is not a unit. Thus $\mu^*(a^\prime)>\mu(a^\prime)$, and so
    \begin{align*}
        \delta_{K_M} \geq \frac{\mu^*(a^\prime)}{\mu(a^\prime)} \geq \frac{\trace_{K_M/\mathbb{Q}}(a^\prime)}{\trace_{K_M/\mathbb{Q}}(a^\prime x^\prime {x^\prime}^*)}>1.
    \end{align*}

    So far we have shown that if $N=sp^n, M=sp^m$ with $m \geq n \geq 0$ and $\gcd(s,p)=1$, $K_N$ not being unit reducible implies that $K_M$ is not unit reducible. We now want to focus on the more general case where $N \mid M$. Let $M=\prod_{i=1}^{t} p_i^{e_i}$, where $p_i$ are prime and $e_i \geq 0$. Then if $N \mid M$, $N=\prod_{i=1}^{t}p_i^{f_i}$ where $e_i \geq f_i \geq 0$. However, it can be deduced that $K_M$ will not be unit reducible if $K_N$ is not unit reducible, since any cyclotomic field of conductor $N^\prime = p_j^{f_j+k} \prod_{i=1: i \neq j}^t p_i^{f_i}$ for any $1 \leq j \leq t$, $k \geq 0$ will also not be unit reducible, and so applying this argument inductively, the cyclotomic field of conductor $M$ is not unit reducible. The inequality $\delta_{K_M} \geq \delta_{K_N}$ in this case follows similarly.
\end{proof}
\begin{proof}[Proof of Theorem \ref{3}]
\begin{table}[]
\centering
\begin{tabular}{|l|l|l|l|}
\hline
\textbf{Value of $N$} & \textbf{$\eta_{K_N}$} & \textbf{$\gamma_n$ where $n=[K:\mathbb{Q}]$} & \textbf{$|\Delta_{K_N}|^{1/n}$} \\ \hline
$5$                   & $5$                   & $\sqrt{2}$                                   & $5^{3/4}$                       \\ \hline
$7$                   & $7$                   & $\left(\frac{64}{3}\right)^{1/6}$            & $7^{5/6}$                       \\ \hline
$8$                   & $2$                   & $\sqrt{2}$                                   & $4$                             \\ \hline
$9$                   & $3$                   & $\left(\frac{64}{3}\right)^{1/6}$            & $3^{3/2}$                       \\ \hline
$12$                  & $4$                   & $\sqrt{2}$                                   & $2\sqrt{3}$                     \\ \hline
$15$                  & $16$                  & $2$                                          & $5^{3/4}\sqrt{3}$               \\ \hline
\end{tabular}
\captionof{table}{A table of values for constants associated to $K_N$.}\label{test}
\end{table}
    Throughout the first part of the proof, we will make use of the values obtained in Table \ref{test}. The cyclotomic fields of conductors $3,4$ are trivially unit reducible. Note that by Theorem 2.2 from \cite{L05}, $\gamma_K \leq \gamma_n |\Delta_K|^{1/n}$ for any field of degree $n$ over $\mathbb{Q}$, and so any field satisfying $\gamma_n |\Delta_K|^{1/n}< n \eta_K^{1/n}$ is strongly unit reducible by Theorem \ref{1}. Referring to Table \ref{test}, we can immediately see that the strict inequality holds for the fields $K_N$ where $N=5,7,12,15$ which implies these fields are strongly unit reducible.
    
    When $N=8$, since $\eta_{K_8}=2$ we have $\gamma_4|\Delta_{K_8}|^{1/4}=4\eta_{K_8}^{1/2}$. However, note that Theorem \ref{1} fails to prove that $K_8$ is unit reducible in this case only if we have $[K_8:\mathbb{Q}]\norm_{K_8/\mathbb{Q}}(axx^*)^{1/4}=\trace_{K_8/\mathbb{Q}}(axx^*)$, for which we would need $axx^*$ to be fixed by all the Galois conjugates, i.e. $axx^* \in \mathbb{Q}$. Since $|\norm_{K_8/\mathbb{Q}}(x)|=2$ is achieved when $x=1+\zeta_8$ (or an associate of this element, or a conjugate-associate of this element, though all conjugates of $1+\zeta_8$ are associates of $1+\zeta_8$), we let $a=q(1+\zeta_8)^{-1}(1+\zeta_8^*)^{-1}$, where $q$ is some positive rational number. Performing HKZ reduction on the real quadratic form $\trace_{K_8/\mathbb{Q}}(axx^*)$ yields $\mu(a)=\trace_{K_8/\mathbb{Q}}(a)=4q$, and so this form is unit reducible, so $K_8$ is unit reducible. However, $\trace_{K_8/\mathbb{Q}}(a)=\trace_{K_8/\mathbb{Q}}(a(1+\zeta_8)(1+\zeta_8)^*)=\mu(a)$, and since $1+\zeta_8$ is not a unit, this means $K_8$ is weakly unit reducible.
    
    Similarly when $N=9$ we have $\gamma_6|\Delta_{K_9}|^{1/6}=6\eta_{K_9}^{1/3}$, where $\eta_{K_9}=3$. Again, Theorem \ref{1} fails to prove that $K_9$ is unit reducible in this case only if we have $[K_9:\mathbb{Q}]\norm_{K_9/\mathbb{Q}}(axx^*)^{1/4}=\trace_{K_9/\mathbb{Q}}(axx^*)$, which occurs when $a=q/xx^*$ and $x=1+\zeta_9+\zeta_9^3$. Performing HKZ reduction on the quadratic form yields $\mu(a)=\trace_{K_9/\mathbb{Q}}(a)=6q$, and so this form is unit reducible, so $K_9$ is unit reducible. Again, though, $\trace_{K_9/\mathbb{Q}}(a)=\trace_{K_9/\mathbb{Q}}(a(1+\zeta_9+\zeta_9^3)(1+\zeta_9+\zeta_9^3)^*)=\mu(a)$, and since $1+\zeta_9+\zeta_9^3$ is not a unit, this means that $K_9$ is weakly unit reducible.

    From now on, we will use the notation
\begin{align*}
    \rho_N(x) \triangleq \trace_{K_N/\mathbb{Q}}(xx^*),
\end{align*}

for any $x \in K_N$ for any conductor $N$.

Consider now the field $K_{2^n}$ for $n \geq 4$. Let $a=((1+\zeta_{2^n})(1+\zeta_{2^n}^{-1}))^{-1}$. Clearly $a \in {K_{2^n}}_{>>0}$.

For any $x=\sum_{i=0}^{2^{n-1}-1} x_i \zeta_{2^n}^i \in \mathcal{O}_{K_{2^n}}$ with $x_i \in \mathbb{Z}$,
\begin{align*}
    \frac{x}{1+\zeta_{2^n}}=\frac{\sum_{i=0}^{2^{n-1}-1} \alpha_i \zeta_{2^n}^i}{1+\zeta_{2^n}}+y,
\end{align*}
for some $y \in \mathcal{O}_{K_{2^n}}$ and $\alpha_i \in \{0,1\}$, by virtue of the fact that $\frac{2}{1+\zeta_{2^n}} \in \mathcal{O}_{K_{2^n}}$. Now, given that $\frac{1+\zeta_{2^n}^i}{1+\zeta_{2^n}} \in \mathcal{O}_{K_{2^n}}$ for all integers $i$, we actually have that either $\frac{x}{1+\zeta_{2^n}} \in \mathcal{O}_{K_{2^n}}$, or
\begin{align*}
    \frac{x}{1+\zeta_{2^n}}=\frac{\zeta_{2^n}^k}{1+\zeta_{2^n}}+y,
\end{align*}
for some $y \in \mathcal{O}_{K_{2^n}}$ and some integer $k$. 

If $\frac{x}{1+\zeta_{2^n}}$ is an element of $\mathcal{O}_{K_{2^n}}$, this means that $x$ cannot be a unit, so we assume $\frac{x}{1+\zeta_{2^n}} \not\in \mathcal{O}_{K_{2^n}}$. Now, since
\begin{align*}
    \rho_{2^n}\left(\sum_{i=0}^{2^{n-1}-1} z_i \zeta_{2^n}^i\right)=2^{n-1}\sum_{i=0}^{2^{n-1}}z_i^2
\end{align*}
for any $z_i \in \mathbb{Q}$, and
\begin{align*}
    \frac{\zeta_{2^n}^k}{1+\zeta_{2^n}}=\zeta_{2^n}^k\sum_{i=0}^{2^{n-1}-1}\frac{(-1)^i}{2}\zeta_{2^n}^i,
\end{align*}
it immediately holds that
\begin{align*}
    \rho_{2^n}\left(\frac{\zeta_{2^n}^k}{1+\zeta_{2^n}}+y\right) \geq \rho_{2^n}\left(\frac{1}{1+\zeta_{2^n}}\right),
\end{align*}
(since $\pm 1/2$ cannot be rounded by an integer to be any smaller), and so for any unit $u$,
\begin{align*}
    \trace_{K_{2^n}/\mathbb{Q}}(auu^*) \geq \trace_{K_{2^n}/\mathbb{Q}}(a),
\end{align*}
  i.e. $a$ is reduced. But since $\rho(1)=2^{n-1}$ and $\rho\left(\frac{1}{1+\zeta_{2^n}}\right)=2^{2n-4}$,
  \begin{align*}
      \rho_{2^n}(1)<\rho_{2^n}\left(\frac{1}{1+\zeta_{2^n}}\right)
  \end{align*}
  for all $n \geq 4$, which implies
  \begin{align*}
      \trace_{K_{2^n}/\mathbb{Q}}(a(1+\zeta_{2^n})(1+\zeta_{2^n}^{-1}))<\trace_{K_{2^n}/\mathbb{Q}}(a) =\mu^*(a).
  \end{align*}
  Hence the field cannot be unit reducible. More specifically, we have
  \begin{align*}
      \frac{\mu^*(a)}{\mu(a)}=2^{n-3},
  \end{align*}
  and so 
  \begin{align*}
      \delta_{K_{2^n}} \geq 2^{n-3}>1,
  \end{align*}
  for all $n \geq 4$.

    Now consider the field $K_{p^n}$, for some odd prime $p$, and again take $a=((1-\zeta_{p^n})(1-\zeta_{p^n}^{-1}))^{-1}$. For any $x=\sum_{i=0}^{(p-1)p^{n-1}} x_i \zeta_{p^n}^i$ where $x_i \in \mathbb{Z}$,
    \begin{align*}
        x&=x_0+x_1\zeta_{p^n}+\sum_{i=2}^{(p-1)p^{n-1}}x_i\zeta_{p^n}^i
        \\&=x_0(1-\zeta_{p^n})+(x_0+x_1)\zeta_{p^n}+x_2\zeta_{p^n}^2+\sum_{i=3}^{(p-1)(p-1)}x_i\zeta_{p^n}^i
        \\&=\dots = (1-\zeta_{p^n})\sum_{i=0}^{p^{n-1}(p-1)-1}y_i\zeta_{p^n}^i + y_{p^{n-1}(p-1)}\zeta_{p^n}^{p^{n-1}(p-1)},
    \end{align*}
    for some $y_i \in \mathbb{Z}$. Using the fact that $\frac{p}{1-\zeta_{p^n}} \in \mathcal{O}_{K_{p^n}}$ yields
    \begin{align*}
        \frac{x}{1-\zeta_{p^n}}=y+\frac{\alpha \zeta_{p^n}^{p^{n-1}(p-1)}}{1-\zeta_{p^{n}}}
    \end{align*}
for some $y \in \mathcal{O}_{K_{p^n}}$ and $\alpha \in \{0,\dots,p-1\}$. Once again, we assume that $\alpha \neq 0$, as we are only interested in the case where $x$ is a unit. Then since
\begin{align*}
    \frac{\zeta_{p^{n}}^{p^{n-1}}}{1-\zeta_{p^n}}=\frac{\zeta_{p^n}^{p^{n-1}}}{p}\sum_{i=0}^{p^{n-1}-1}\sum_{j=0}^{p-2}(p-1-j)\zeta_{p^n}^{i+jp^{n-1}},
\end{align*}
this means that
\begin{align*}
    \frac{\alpha\zeta_{p^{n}}^{p^{n-1}}}{1-\zeta_{p^n}}=\frac{\zeta_{p^n}^{p^{n-1}}}{p}\sum_{i=0}^{p^{n-1}-1}\sum_{j=0}^{p-2}\alpha(p-1-j)\zeta_{p^n}^{i+jp^{n-1}}.
\end{align*}
Since $\alpha \neq 0 \mod p$, $\alpha(p-1-i)$ permutes each $(p-1-i)$ to a unique element mod $p$ for each distinct $0 \leq i \leq p-2$, so
\begin{align*}
    \frac{x}{1-\zeta_{p^n}}=y+\frac{\zeta_{p^n}^{p^{n-1}}}{p}\sum_{i=0}^{p^{n-1}-1}\sum_{j=0}^{p-2}(p-1-j)\zeta_{p^n}^{i+r(j)p^{n-1}},
\end{align*}
where $r$ denotes some permutation of the set $\{0,\dots,p-2\}$, and $y \in \mathcal{O}_{K_{p^n}}$.

Now, for any $z=\sum_{i=0}^{p^{n-1}(p-1)-1} z_i \zeta_{p^n}^i$, a simple computation yields
\begin{align*}
    \rho_{p^n}(z)=p^{n-1}\sum_{i=0}^{p^{n-1}-1}Q(z_{i},z_{i+p^{n-1}},z_{i+2p^{n-1}}, \dots, z_{i+(p-2)p^{n-1}}),
\end{align*}
   where $Q$ defines the $p-1$-dimensional positive-definite quadratic form
   \begin{align*}
       Q(m_1,\dots,m_{p-1})=(p-1)\sum_{i=1}^{p-1}m_i^2 - 2\sum_{i<j}m_im_j.
   \end{align*}
   It was shown in \cite{L75} that 
   \begin{align*}
       Q\left(\frac{r(p-1)}{p}-m_1,\frac{r(p-2)}{p}-m_2,\dots,\frac{r(1)}{p}-m_{p-1}\right) \geq Q\left(\frac{r(p-1)}{p},\frac{r(p-2)}{p},\dots,\frac{r(1)}{p}\right),
   \end{align*}
   for all integers $m_1,\dots,m_{p-1}$ and permutations $r$ of the set $\{1,\dots,p-1\}$. Collectively, this means that
   \begin{align*}
       \rho_{p^n}\left(\frac{x}{1-\zeta_{p^{n}}}\right) \geq \rho_{p^n}\left(\frac{1}{1-\zeta_{p^{n}}}\right),
   \end{align*}
   for all $x \in \mathcal{O}_{K_{p^n}}$ such that $1-\zeta_{p^n}$ does not divide $x$, and so
   \begin{align*}
       \trace_{K_{p^n}/\mathbb{Q}}(auu^*) \geq \trace_{K_{p^n}/\mathbb{Q}}(a)
   \end{align*}
   for all units $u$, i.e. $a$ is reduced. Also note that
   \begin{align*}
       \trace_{K_{p^n}/\mathbb{Q}}(a)=p^{2(n-1)}\frac{p^2-1}{12}.
   \end{align*}
   However, we also have
   \begin{align*}
       \trace_{K_{p^n}/\mathbb{Q}}(a(1-\zeta_{p^n})(1-\zeta_{p^n}^{-1}))=\trace_{K_{p^n}/\mathbb{Q}}(1)=p^{n-1}(p-1)=\mu(a),
   \end{align*}
   and so
   \begin{align*}
       \delta_{K_{p^n}} \geq \frac{\trace_{K_{p^n}/\mathbb{Q}}(a)}{\mu(a)}=\frac{\mu^*(a)}{\mu(a)}=p^{n-1}\frac{p+1}{12}>1,
   \end{align*}
   if $p=3$ and $n \geq 3$, $p=5$ and $n \geq 2$, $p=7$ and $n \geq 2$, $p=11$ and $n \geq 2$ or $p \geq 13$.
\end{proof}
\begin{proof}[Proof of Corollary \ref{cor2}]
    It suffices to show that for each positive $1 \leq \delta < \infty$, there can only be finitely many integers $M$ satisfying $\delta_{K_M} \leq \delta$. Suppose that $M=2^e\prod_{i=1}^t p_i^{e_i}$, for some odd primes $p_1,\dots,p_t$ ordered so $p_i<p_{i+1}$ for all $1 \leq i \leq t-1$, and positive integers $e,e_1,\dots,e_t$. Recall that Theorem \ref{2} states that for any $N$ dividing $M$, $\delta_{K_N} \leq \delta_{K_M}$. Since $\delta_{K_{p^n}} \geq p^{n-1}\frac{p+1}{12}\geq \frac{p+1}{12}$ for any odd prime $p$ and integer $n \geq 1$, if $\delta_{K_M} \leq \delta$ we must have $p_t \leq 12\delta -1$. Moreover, since $ \delta \geq \delta_{K_M} \geq \delta_{K_{p_i^{e_i}}} \geq p^{e_i-1}\frac{p_i+1}{12}>\frac{p^{e_i}}{12}$ for each $1 \leq i \leq t$, this implies $e_i < \frac{\log(\delta)}{\log(p_i)}+\log(12)$ for each $e_i$. Similarly, $e \leq \frac{\log(\delta)}{\log(2)}+4$. Therefore if $M$ satisfies $\delta_{K_M} \leq \delta$ we must have
    \begin{align*}
        M < 2^{\frac{\log(\delta)}{\log(2)}+4}\prod_{i=1}^t p_i^{\frac{\log(\delta)}{\log(p_i)}+\log(12)},
    \end{align*}
    satisfying $p_t \leq 12\delta-1$. Clearly there are only finitely many integers that satisfy this inequality if $\delta$ is finite.
\end{proof}
\section{Totally Real Subfields of Cyclotomic Fields}
For any cyclotomic field $K_N$ for some conductor $N$, we will use the notation $K_N^+$ to refer to the maximal totally real subfield of $K_N$, i.e. $K_N^+=\mathbb{Q}(\zeta_N+\zeta_N^{-1})$. A natural question to ask is whether $K_N$ being (not) unit reducible implies that $K_N^+$ is unit reducible and vice versa.

Recall that for every unit $u$ in $K_N$, there exists a unit $v$ in $K_N^+$ such that $u=\zeta_N^i v$ (see e.g. \cite{jsmilne}, Proposition 6.7). For any field $K$ and any subset $S$, denote by $S^2=\{xx^*: x \in S\}$. It then follows that ${\mathcal{O}_{K_N}^*}^2={\mathcal{O}_{K_N^+}^*}^2$. We also have $\mathcal{O}_{K_N^+}^2 \subseteq \mathcal{O}_{K_N}^2$.

The set of totally positive elements in $K_N$ and $K_N^+$ are identical (since necessarily the set of totally positive elements must be a subset of the totally real subfield). Then for any $a \in {K_N^+}_{>>0}$ and its associated element $a^\prime$ after being lifted into $K_N$ (notice every element of ${K_N}_{>>0}$ can be constructed in this manner), $\mu^*(a)=\frac{1}{2}\mu^*(a^\prime)$ (where the $\frac{1}{2}$ factor comes from the fact that $[K_N:K_N^+]=2$). Since $\mathcal{O}_{K_N^+}^2 \subseteq \mathcal{O}_{K_N}^2$, $\frac{1}{2}\mu(a^\prime) \leq \mu(a)$. From this we can deduce that if $K_N$ is unit reducible, then $K_N^+$ is unit reducible, and if $K_N^+$ is not unit reducible then $K_N$ is not unit reducible.

From this, combined with Theorem \ref{3}, we deduce that $K_N^+$ is unit reducible for $N=3,4,5,7,8,9,12,15$. However, of these the result that $K_N^+$ for $N=15$ is unit reducible is the only new result. The maximal totally real subfields of $K_3, K_4$ are trivially unit reducible (as they are both $\mathbb{Q}$), whilst the maximal totally real subfields of $K_N$ are real quadratic for $N=5,8,12$ and simplest cubic fields for $N=7,9$, the cases of which were both covered in \cite{LPL23}.

We now deduce a lower bound on $\delta_{K_N^+}$ using similar techniques in the last section.
\begin{proposition}
    For all $n \geq 1$,
    \begin{align*}
        \delta_{K_{2^n}^+} \geq 2^{n-4},
    \end{align*}
    and for all odd primes $p$,
    \begin{align*}
        \delta_{K_{p^n}^+} \geq \frac{p^{n-1}(p^2-1)}{24(p-2)}.
    \end{align*}
\end{proposition}
\begin{proof}
    From now on, we will use the notation $\theta_N \triangleq \zeta_N+\zeta_N^{-1}$. Consider the field $K_{2^n}^+$ for some $n \geq 1$. We set $a=(2+\theta_{2^n})^{-1}$, which is totally positive for any $n$. Lifting $a$ to $K_{2^n}$ as $a^\prime=(1+\zeta_{2^n})^{-1}(1+\zeta_{2^n}^{-1})$, we have already shown that $\mu^*(a^\prime)=\trace_{K_{2^n}/\mathbb{Q}}(a^\prime)=2^{2n-4}$, and so $\mu^*(a)=2^{2n-5}$. However, $a(2+\theta_{2^n})^2=2+\theta_{2^n}=a^{-1}$, and since $\trace_{K_{2^n}/\mathbb{Q}}({a^{\prime}}^{-1})=2^{n}$ we have $\trace_{K_{2^n}^+/\mathbb{Q}}(a(2+\theta_{2^n})^2)=\trace_{K_{2^n}^+/\mathbb{Q}}(a^{-1})=2^{n-1} \geq \mu(a)$. Hence
    \begin{align*}
        \delta_{K_{2^n}^+} \geq \frac{\mu^*(a)}{\mu(a)} \geq \frac{\mu^*(a)}{\trace_{K_{2^n}^+/\mathbb{Q}}(a^{-1})}=2^{n-4}.
    \end{align*}
    Similarly, when we consider the field $K_{p^n}^+$ for some odd prime $p$ and $n \geq 1$, we take $a=2-\theta_{p^n}$, which is totally positive for any $n$. Following similar logic to before, and using the fact that $\trace_{K_{p^n}/\mathbb{Q}}((1-\zeta_{p^n})(1-\zeta_{p^n}^{-1}))=2p^{n-1}(p-2)$,
    \begin{align*}
        \delta_{K_{p^n}^+} \geq \frac{\mu^*(a)}{\mu(a)} \geq \frac{\mu^*(a)}{\trace_{K_{p^n}^+/\mathbb{Q}}(a^{-1})}=\frac{p^{n-1}(p^2-1)}{24(p-2)}.
    \end{align*}
\end{proof}
\begin{proposition}
    If $N \mid M$, if $K_N^+$ is not unit reducible then $K_M^+$ is not unit reducible. Moreover, $\delta_{K_M^+} \geq \delta_{K_N^+}$.
\end{proposition}
    \begin{proof}
        The proof follows almost identically to the proof of Theorem \ref{2}, using the fact that $\mathcal{O}_{K_M^+}=\mathbb{Z}[\theta_M]$.
    \end{proof}
We immediately obtain the following corollaries as a result, and omit the proofs as they are identical to that of the cyclotomic case.
\begin{corollary}
    The field $K_N^+$ is not unit reducible if $N$ is divisible by $2^5$, $3^3$, $5^2$, $7^2$, $11^2$, $13^2$, $17^2$, $19^2$, or any odd prime $p \geq 23$.
\end{corollary}
\begin{corollary}
    $\delta_{K_N^+} \to \infty$ as $N \to \infty$.
\end{corollary}
\section{Concluding Remarks}
In this work, we have deduced a number of results regarding unit reducibility in cyclotomic fields and their totally real subfields. Specifically, we determined that the number of unit reducible cyclotomic fields must be finite in number, and that the discrepancy between the trace of the shortest generator of a positive unary form and the shortest nonzero vector length of the corresponding positive-definite quadratic form can become arbitrarily large, as the conductor for the cyclotomic field grows. We also showed a similar result for their maximal totally real subfields. We deduced a simple method by which we can determine whether an arbitrary field is unit reducible, and used this to prove that a number of cyclotomic fields are unit reducible, though this list is not complete. 

Of course, the principal open problem is to determine the entire list of unit reducible cyclotomic fields. The authors make the following conjecture:
\begin{conjecture}
    The cyclotomic fields $K_N$ where $N=3,4,5,7,8,9,11,12,15,20,21,24$ are the only unit reducible cyclotomic fields.
\end{conjecture}
This would be an interesting result if true - in \cite{S00}, the number of homothety classes of perfect unary forms in cyclotomic fields were studied. Of these, the fields of conductors $N=3,4,5,7,8,9,11,20,21,24$ were found to have all the perfect unary forms (up to homothety class) correspond to perfect rational quadratic forms of dimension $\phi(N)$. Moreover, none of the other cyclotomic fields considered satisfied this property. Therefore, if it were shown that this list of fields were the only unit reducible fields, this may suggest a connection between the studies of unit reducible fields and perfect unary forms for cyclotomic fields.

\end{document}